\documentclass[reqno,a4paper]{amsart}
\usepackage{amssymb,url}
\usepackage[hypertex]{hyperref}
\allowdisplaybreaks[4]
\date{}
\theoremstyle{plain}
\newtheorem{thm}{Theorem}
\theoremstyle{remark}
\newtheorem{rem}{Remark}
\DeclareMathOperator{\td}{d\mspace{-1mu}}

\begin{document}

\title[A class of logarithmically completely monotonic functions]
{More supplements to a class of logarithmically completely monotonic functions associated with the gamma function}

\author[S. Guo]{Senlin Guo}
\address[S. Guo]{Department of Mathematics, Zhongyuan University of Technology, Zhengzhou City, Henan Province, 450007, China}
\email{\href{mailto: S. Guo <sguo@hotmail.com>}{sguo@hotmail.com}, \href{mailto: S. Guo <senlinguo@gmail.com>}{senlinguo@gmail.com}}

\author[F. Qi]{Feng Qi}
\address[F. Qi]{Research Institute of Mathematical Inequality Theory, Henan Polytechnic University, Jiaozuo City, Henan Province, 454010, China}
\email{\href{mailto: F. Qi <qifeng618@gmail.com>}{qifeng618@gmail.com}, \href{mailto: F. Qi <qifeng618@hotmail.com>}{qifeng618@hotmail.com}, \href{mailto: F. Qi <qifeng618@qq.com>}{qifeng618@qq.com}}
\urladdr{\url{http://qifeng618.spaces.live.com}}

\begin{abstract}
In this article, a necessary and sufficient condition and a necessary condition are established for a function involving the gamma function to be logarithmically completely monotonic on $(0,\infty)$. As applications of the necessary and sufficient condition, some inequalities for bounding the psi and polygamma functions and the ratio of two gamma functions are derived.
\par
This is a continuator of the paper \cite{Guo-Qi-Srivastava2007-02.tex}.
\end{abstract}



\thanks{The second author was supported partially by the China Scholarship Council}

\thanks{This paper was typeset using \AmS-\LaTeX}

\maketitle


\section{Introduction}

Recall \cite{Atanassov, minus-one-rgmia} that a positive function $f$ is said to be logarithmically completely monotonic on an interval $I$ if $f$ has derivatives  of all orders on $I$ and
\begin{equation}
(-1)^n[\ln f(x)]^{(n)}\ge0.
\end{equation}
This kind of functions has very closer relationships with the Laplace transforms, Stieltjes transforms and infinitely divisible completely monotonic functions. For more detailed information, please refer to \cite{CBerg, grin-ismail, Guo-Qi-Srivastava2007.tex, Guo-Qi-Srivastava2007-02.tex, sandor-gamma-2-ITSF.tex, compmon2, auscm} and related references therein.
\par
It is well-known that the classical Euler gamma function is defined for $x>0$ by
\begin{equation}\label{egamma}
\Gamma(z)=\int^\infty_0t^{x-1} e^{-t}\td t.
\end{equation}
The logarithmic derivative of $\Gamma(z)$, denoted by $\psi(z)=\frac{\Gamma'(z)}{\Gamma(z)}$, is called the psi function, and $\psi^{(k)}$ for $k\in\mathbb{N}$ are called the polygamma functions.
\par
For $\alpha\in\mathbb{R}$ and $\beta\ge0$, define
\begin{equation}\label{ffabrx}
f_{\alpha,\beta,\pm1}(x)=\biggl[\frac{e^x\Gamma(x+\beta)}{x^{x+\beta-\alpha}}\biggr]^{\pm1},\quad
x\in(0,\infty).
\end{equation}
The investigation of the function $f_{\alpha,\beta,\pm1}(x)$ has a long history. In what follows, we would like to give a short survey in the literature.
\par
In \cite[Theorem~1]{keckic}, for showing
\begin{equation}\label{kec}
\frac{b^{b-1}}{a^{a-1}} e^{a-b}<\frac{\Gamma(b)}{\Gamma(a)}
<\frac{b^{b-1/2}}{a^{a-1/2}} e^{a-b}
\end{equation}
for $b>a>1$, monotonic properties of the functions $\ln f_{\alpha,0,+1}(x)$ and $\ln f_{\alpha,0,+1}(x)$ on $(1,\infty)$ were obtained.
\par
In \cite[Theorem~2.1]{Muldoon-78}, the function $f_{\alpha,0,+1}(x)$ for $\alpha\le1$ was proved to be logarithmically completely monotonic on $(0,\infty)$. In~\cite[Theorem~2.1]{Ismail-Lorch-Muldoon}, the functions $f_{\alpha,0,+1}(x)$ and $f_{\alpha,0,-1}(x)$ were proved to be logarithmically completely monotonic on $(0,\infty)$ if and only if $\alpha\le\frac12$ and $\alpha\ge1$ respectively. These results were mentioned in \cite[Theorem~2.1]{Ismail-Muldoon-119} later. However, we do not think the proof in \cite{Ismail-Lorch-Muldoon} for the necessity is convincible.
\par
In~\cite[Theorem~3.2]{Anderson}, it was recovered that the function $f_{1/2,0,+1}(x)$ is decreasing and logarithmically convex from $(0,\infty)$ onto $\big(\sqrt{2\pi}\,,\infty\big)$ and that the function $f_{1,0,+1}(x)$ is increasing and logarithmically concave from $(0,\infty)$ onto $(1,\infty)$.
\par
In \cite[p.~376, Theorem~2]{psi-alzer}, it was presented that the function $f_{\alpha,0,+1}(x)$ is decreasing on $(c,\infty)$ for $c\ge0$ if and only if $\alpha\le\frac12$ and increasing on $(c,\infty)$ if and only if
\begin{equation}
\alpha\ge
\begin{cases}
c[\ln c-\psi(c)]&\text{if $c>0$},\\1&\text{if $c=0$}.
\end{cases}
\end{equation}
\par
The necessary and sufficient conditions for the functions $f_{\alpha,0,+1}(x)$ and $f_{\alpha,0,-1}(x)$ to be logarithmically completely monotonic on $(0,\infty)$, stated in \cite[Theorem~2.1]{Ismail-Lorch-Muldoon} and \cite[Theorem~2.1]{Ismail-Muldoon-119} and mentioned above, were recovered in~\cite[Theorem~2]{chen-qi-log-jmaa}. Moreover, the function $f_{\alpha,\beta,+1}(x)$ was proved in~\cite[Theorem~1]{chen-qi-log-jmaa} to be logarithmically completely monotonic on $(0,\infty)$ if $2\alpha \leq 1 \leq \beta$. Using monotonic properties of the functions $f_{1/2,0,+1}(x)$ and $f_{1,0,-1}(x)$, the inequality~\eqref{kec} was extended in~\cite[Remark~1]{chen-qi-log-jmaa} from $b>a>1$ to $b>a>0$.
\par
In \cite{chen-ASCM-07}, after proving once again the logarithmically completely monotonic property of the functions $f_{1/2,0,+1}(x)$ and $f_{1,0,-1}(x)$, in virtue of Jensen's inequality for convex functions, the upper and lower bounds were established: For positive numbers $x$ and $y$, the inequality
\begin{equation}\label{chen-gurland-ineq}
\frac{x^{x-1/2}y^{y-1/2}}{[(x+y)/2]^{x+y-1}} \le \frac{\Gamma(x)\Gamma(y)}{[\Gamma((x+y)/2)]^2} \le\frac{x^{x-1}y^{y-1}}{[(x+y)/2]^{x+y-2}}
\end{equation}
holds, where the middle term in \eqref{chen-gurland-ineq} is called Gurland's ratio \cite{Merkle-Gurland}. The left-hand side inequality in \eqref{chen-gurland-ineq} is same as the corresponding one in \cite[Theorem~1]{Merkle-Gurland}, but their upper bounds do not include each other.
\par
Recently, some new conclusions on logarithmically completely monotonic properties of the function $f_{\alpha,\beta,+1}(x)$ were procured in \cite[Theorem~1]{Guo-Qi-Srivastava2007-02.tex}:
\begin{enumerate}
\item
If $\beta\in(0,\infty)$ and $\alpha\le0$, then $f_{\alpha,\beta,+1}(x)$ is logarithmically completely monotonic on $(0,\infty)$;
\item
If $\beta\in(0,\infty)$ and $f_{\alpha,\beta,+1}(x)$ is a logarithmically completely monotonic function on $(0,\infty)$, then $\alpha\le\min\bigl\{\beta,\frac12\bigr\}$;
\item
If $\beta \ge 1$, then $f_{\alpha,\beta,+1}(x)$ is logarithmically completely monotonic on $(0,\infty)$ if and only if $\alpha\le\frac12$.
\end{enumerate}
As direct consequences of these monotonic properties, it is deduced immediately that if $x$ and $y$ are positive numbers with $x\ne y$, then
\begin{enumerate}
\item
the inequality
\begin{equation}\label{guo-neces-suff-ineq}
\frac{\Gamma(x+\beta)}{\Gamma(y+\beta)} <\frac{x^{x+\beta-\alpha}}{y^{y+\beta-\alpha}}e^{1/(y-x)}
\end{equation}
for $\beta\ge1$ and $x>y>0$ holds true if and only if $\alpha\le\frac12$;
\item
the inequality \eqref{guo-neces-suff-ineq} for $\beta\in(0,\infty)$ holds true also if $\alpha\le0$.
\end{enumerate}

In this paper, for continuing the work in the paper~\cite{Guo-Qi-Srivastava2007-02.tex}, we consider logarithmically completely monotonic properties of the function $f_{\alpha,\beta,-1}(x)$ on $(0,\infty)$.
\par
The main results of this paper are as follows.

\begin{thm}\label{fth1}
If the function $f_{\alpha,\beta,-1}(x)$ is logarithmically completely monotonic on $(0,\infty)$, then either $\beta>0$ and $\alpha\ge \max\bigl\{\beta,\frac12\bigr\}$ or $\beta=0$ and $\alpha\ge 1$.
\end{thm}

\begin{thm}\label{fth3}
If $\beta\ge\frac12$, the necessary and sufficient condition for the function $f_{\alpha,\beta,-1}(x)$ to be logarithmically completely monotonic on $(0,\infty)$ is $\alpha\ge\beta$.
\end{thm}

As the first application of Theorem~\ref{fth3}, the following inequalities are derived by using logarithmically completely monotonic properties of the function $f_{\alpha,\beta,\pm1}(x)$ on $(0,\infty)$.

\begin{thm}\label{kevic-type-ineq}
Let $\beta$ be a positive number.
\begin{enumerate}
\item
For $k\in\mathbb{N}$, double inequalities
\begin{equation}\label{qi-psi-ineq-1}
\ln x-\frac1x\le\psi(x)\le\ln x-\frac1{2x}
\end{equation}
and
\begin{equation}\label{qi-psi-ineq}
\frac{(k-1)!}{x^k}+\frac{k!}{2x^{k+1}}\le(-1)^{k+1}\psi^{(k)}(x)
\le\frac{(k-1)!}{x^k}+\frac{k!}{x^{k+1}}
\end{equation}
hold in $(0,\infty)$.
\item
When $\beta>0$, inequalities
\begin{equation}\label{beta>0-1}
\psi(x+\beta)\le \ln x+\frac{\beta}x
\end{equation}
and
\begin{equation}\label{beta>0-2}
(-1)^{k}\psi^{(k-1)}(x+\beta)\ge \frac{(k-2)!}{x^{k-1}}-\frac{\beta(k-1)!}{x^{k}}
\end{equation}
hold on $(0,\infty)$ for $k\ge2$.
\item
When $\beta\ge\frac12$, inequalities
\begin{equation}\label{beta>1/2-dou-ineq}
\psi(x+\beta)\ge \ln x\quad \text{and}\quad (-1)^{k}\psi^{(k-1)}(x+\beta)\le \frac{(k-2)!}{x^{k-1}}
\end{equation}
hold on $(0,\infty)$ for $k\ge2$.
\item
When $\beta\ge1$, inequalities
\begin{equation}\label{beta>1/2-dou-ineq=1}
\psi(x+\beta)\le \ln x+\frac{\beta-1/2}x
\end{equation}
and
\begin{equation}\label{beta>1/2-dou-ineq=2}
(-1)^{k}\psi^{(k-1)}(x+\beta)\ge \frac{(k-2)!}{x^{k-1}}-\frac{(\beta-1/2)(k-1)!}{x^{k}}
\end{equation}
holds on $(0,\infty)$ for $k\ge2$.
\end{enumerate}
\end{thm}

As the second application of Theorem~\ref{fth3}, the following inequalities are derived by using logarithmically convex properties of the function $f_{\alpha,\beta,\pm1}(x)$ on $(0,\infty)$.

\begin{thm}\label{gurland-deriv-thm}
Let $n\in\mathbb{N}$, $x_k>0$ for $1\le k\le n$, $p_k\ge0$ satisfying $\sum_{k=1}^np_k=1$. If either $\beta>0$ and $\alpha\le0$ or $\beta\ge1$ and $\alpha\le\frac12$, then
\begin{equation}\label{n-gurland-ineq}
\frac{\prod_{k=1}^n[\Gamma(x_k+\beta)]^{p_k}}{\Gamma\bigl(\sum_{k=1}^np_kx_k+\beta\bigr)} \ge\frac{\prod_{k=1}^nx_k^{p_k(x_k+\beta-\alpha)}} {\bigl(\sum_{k=1}^np_kx_k\bigr)^{\sum_{k=1}^np_kx_k+\beta-\alpha}}.
\end{equation}
If $\alpha\ge\beta\ge\frac12$, then the inequality \eqref{n-gurland-ineq} reverses.
\end{thm}

As the final application of Theorem~\ref{fth3}, the following inequality may be derived by using the decreasingly monotonic property of the function $f_{\alpha,\beta,-1}(x)$ on $(0,\infty)$.

\begin{thm}\label{gurland-deriv-thm-mon}
If $\alpha\ge\beta\ge\frac12$, then
\begin{equation}\label{guo-neces-suff-ineq}
I(x,y)<\biggl[\biggl(\frac{x}y\biggr)^{\alpha-\beta} \frac{\Gamma(x+\beta)}{\Gamma(y+\beta)}\biggr]^{1/(x-y)}
\end{equation}
holds true for $x,y\in(0,\infty)$ with $x\ne y$, where
\begin{equation}
I(a,b)=\frac 1e\biggl(\frac{b^b}{a^a}\biggr)^{1/(b-a)}
\end{equation}
for $a>0$ and $b>0$ with $a\ne b$ is the identric or exponential mean.
\end{thm}

\begin{rem}
For $\beta\in\mathbb{R}$, let
\begin{equation}\label{guo-funct-g}
h_{\beta,\pm1}(x)=\biggl[\frac{e^x\Gamma(x+1)}{(x+\beta)^{x+\beta}}\biggr]^{\pm1}
\end{equation}
on $(\max\{0,-\beta\},\infty)$. In \cite{s-guo-ijpam, Guo-Qi-Srivastava2007.tex}, it was showed that the functions $h_{\beta,+1}(x)$ and $h_{\beta,-1}(x)$ are logarithmically completely monotonic if and only if $\beta\ge1$ and $\beta\le\frac12$ respectively. As consequences of monotonic results of the function $h_{\beta,\pm1}(x)$, the following two-sided inequality was derived in \cite{Guo-Qi-Srivastava2007.tex}:
\begin{equation}\label{guosl-ineq}
\frac{(x+1)^{x+1}}{(y+1)^{y+1}}\;e^{y-x}<\frac{\Gamma(x+1)}{\Gamma(y+1)} <\frac{(x+1/2)^{x+1/2}}{(y+1/2)^{y+1/2}}\;e^{y-x}
\end{equation}
for $y>x>0$, where the constants $1$ and $\frac12$ in the very left and the very right sides of the two-sided inequality \eqref{guosl-ineq} cannot be replaced, respectively, by smaller and larger numbers.
\end{rem}

\begin{rem}
In \cite{note-on-li-chen.tex}, it was showed that the function
\begin{equation}
h(x)=\frac{e^x\Gamma(x)}{x^{x[1-\ln x+\psi(x)]}}
\end{equation}
on $(0,\infty)$ has a unique maximum $e$ at $x=1$, with the following two limits
\begin{equation}\label{2limits}
\begin{aligned}
\lim_{x\to0^+}h(x)&=1 & \text{and} &&\lim_{x\to\infty}h(x)&=\sqrt{2\pi}\,.
\end{aligned}
\end{equation}
As consequences of the monotonicity of $h(x)$, it was concluded in \cite{note-on-li-chen.tex} that the following inequality
\begin{equation}\label{note-li-chen-ineq}
\frac{x^{x[\ln x-\psi(x)-1]}}{y^{y[\ln y-\psi(y)-1]}} e^{y-x} <\frac{\Gamma(y)}{\Gamma(x)}
\end{equation}
holds true for $y>x\ge1$. If $0<x<y\le1$, the inequality \eqref{note-li-chen-ineq} is reversed.
\end{rem}

\begin{rem}
It is worthwhile to point out that \cite[Thorem~1.3]{sandor-gamma-2-ITSF.tex} is equivalent to necessary and sufficient conditions in \cite[Theorem~2]{chen-qi-log-jmaa} and \cite[Theorem~2.1]{Ismail-Lorch-Muldoon} for the functions $f_{1/2,0,+1}(x)$ and $f_{1,0,-1}(x)$ to be logarithmically completely monotonic on $(0,\infty)$. See also \cite{Guo-Qi-Srivastava2007.tex, Guo-Qi-Srivastava2007-02.tex} and related references therein.
\end{rem}

\section{Proofs of main results}

Now we are in a position to prove our theorems.

\begin{proof}[Proof of Theorem \ref{fth1}]
Suppose that $f_{\alpha,\beta,-1}(x)$ is logarithmically completely monotonic on $(0,\infty)$. Then
\begin{equation}\label{f500}
[\ln f_{\alpha,\beta,-1}(x)(x)]'=\ln x-\psi(x+\beta)+\frac{\beta-\alpha}x \le0,
\end{equation}
from which we have
\begin{equation}\label{f505}
\beta-\alpha \le x[\psi(x+\beta)-\ln x], \quad x\in(0,\infty).
\end{equation}
If $\beta>0$, then
\begin{equation*}\label{f510}
\beta-\alpha \le \lim_{x\to 0^+}[x\psi(x+\beta)-x\ln x]=0.
\end{equation*}
That is
\begin{equation}\label{f515}
\alpha \ge \beta.
\end{equation}
\par
Using the asymptotic formula
\begin{equation}\label{flm22}
\psi(x)=\ln x-\frac1{2x}+O\bigg(\frac1{x^2}\bigg), \quad x\to \infty,
\end{equation}
see \cite[p.~47]{er}, in \eqref{f505} for $\beta>0$, we obtain
\begin{align*}
\beta-\alpha&\le\lim_{x\to\infty}x\biggl[\ln(x+\beta)-\frac{1}{2(x+\beta)} +O\biggl(\frac{1}{x^2}\biggr)-\ln x\biggr]\\
\quad &=\lim_{x\to \infty}\biggl[x\ln\biggl(1+\frac{\beta}x\biggr)\biggr]-\frac{1}{2}\\
\quad &=\beta\lim_{x\to \infty}\biggl[\frac{x}{\beta} \ln\biggl(1+\frac{\beta}x\biggr)\biggr]-\frac{1}{2}\\
\quad &=\beta -\frac{1}{2},
\end{align*}
from which we get
\begin{equation}\label{f530}
\alpha \ge \frac{1}{2}.
\end{equation}
Combining \eqref{f515} and \eqref{f530} yields
\begin{equation}\label{f550}
\alpha \ge \max\biggl\{\beta,\frac12\biggr\} \quad \text{if} \quad \beta>0.
\end{equation}
\par
If $\beta=0$, considering $f_{\alpha,0,-1}(x)=f_{\alpha,1,-1}(x)$ and \eqref{f550} yields $\alpha\ge\max\bigl\{1,\frac12\bigr\}=1$. The proof is complete.
\end{proof}

\begin{proof}[Proof of Theorem \ref{fth3}]
By Theorem \ref{fth1}, the necessary condition is obtained readily.
\par
Differentiating \eqref{f500} and making use of
\begin{equation*}
\psi ^{(n)}(x)=(-1)^{n+1}\int_{0}^{\infty}\frac{t^{n}}
{1-e^{-t}}e^{-xt}\td t,\quad x\in (0,\infty)
\end{equation*}
and
\begin{equation*}
\frac1{x^n}=\frac1{\Gamma(n)}\int_0^\infty t^{n-1}e^{-xt}\td t,\quad x\in (0,\infty),
\end{equation*}
see \cite[p.~884]{grads}, gives
\begin{gather}
(-1)^{n}[\ln f_{\alpha,\beta,-1}(x)]^{(n)} =\frac{(n-2)!}{x^{n-1}}-(-1)^{n}\psi^{(n-1)}(x+\beta)-\frac{(\beta-\alpha)(n-1)!}{x^{n}}\notag\\
\begin{aligned}\label{f595}
&=\int_{0}^{\infty}t^{n-2}e^{-xt}\td t-\int_{0}^{\infty}\frac{t^{n-1}}{1-e^{-t}}e^{-(x+\beta)t}\td t
-(\beta-\alpha)\int_{0}^{\infty}t^{n-1}e^{-xt}\td t\\
&=\int_{0}^{\infty} \bigg[\alpha-\beta-\frac1t\bigg(e^{(1-\beta)t}\frac{t}{e^t-1}-1\bigg)\bigg] {t^{n-1}e^{-xt}}\td t
\end{aligned}
\end{gather}
for $n\ge2$. The inequality
\begin{equation}\label{f800}
\frac{t}{e^t-1}< \frac{1}{e^{t/2}}, \quad t\in(0,\infty)
\end{equation}
was ever used in \cite{mathieu-rostock, mathieu-rostock-rgmia, mathieu-ijmms}. Substituting it into \eqref{f595} leads to
\begin{equation*}\label{f801}
(-1)^{n}[\ln f_{\alpha,\beta,-1}(x)]^{(n)}\ge \int_{0}^{\infty}
\bigg[\alpha-\beta-\frac{e^{(1/2-\beta)
t}-1}{t}\bigg]t^{n-1}e^{-xt}\td t,\quad n\ge2.
\end{equation*}
Since the function
\begin{equation}
\frac{e^{(1/2-\beta)t}-1}t,\quad t\in(0,\infty)
\end{equation}
is increasing, if $\alpha \ge \beta \ge \frac12$, then
\begin{equation}\label{f806}
(-1)^{n}[\ln f_{\alpha,\beta,-1}(x)]^{(n)}\ge (\alpha-\beta)
\int_{0}^{\infty} t^{n-1}e^{-xt}\td t \ge 0,\quad n\ge2.
\end{equation}
\par
By using \eqref{flm22}, it follows that
\begin{equation*}
[\ln f_{\alpha,\beta,-1}(x)(x)]'=\ln\biggl(1+\frac{\beta}{x}\biggr) -\frac{1}{2(x+\beta)}+\frac{\alpha-\beta}{x}+O\biggl(\frac{1}{x^2}\biggr)
\end{equation*}
as $x\to\infty$, thus for all $\alpha$ and $\beta$,
\begin{equation}\label{f718}
\lim_{x\to\infty}[\ln f_{\alpha,\beta,-1}(x)(x)]'=0.
\end{equation}
From \eqref{f718} and \eqref{f806}, we have
\begin{equation*}
[\ln f_{\alpha,\beta,-1}(x)]' \le 0 \quad \text{if $\alpha\ge\beta\ge\frac12$}.
\end{equation*}
Thus $(-1)^{n}[\ln f_{\alpha,\beta,-1}(x)]^{(n)}\ge 0$ are valid for all $n\in\mathbb{N}$. The proof is complete.
\end{proof}

\begin{proof}[Proof of Theorem~\ref{kevic-type-ineq}]
If $f_{\alpha,\beta,-1}(x)$ is logarithmically completely monotonic on $(0,\infty)$, then
$$
(-1)^k[\ln f_{\alpha,\beta,-1}(x)]^{(k)}\ge0
$$
on $(0,\infty)$ for $k\in\mathbb{N}$, which is equivalent to
\begin{equation}\label{reason-ineq-1}
\psi(x+\beta)\ge \ln x+\frac{\beta-\alpha}x
\end{equation}
and, for $k\ge2$,
\begin{equation}\label{reason-ineq-2}
(-1)^{k}\psi^{(k-1)}(x+\beta)\le \frac{(k-2)!}{x^{k-1}}-\frac{(\beta-\alpha)(k-1)!}{x^{k}}.
\end{equation}
Hence, Theorem~\ref{fth3} implies inequalities in \eqref{beta>1/2-dou-ineq}.
\par
If $\beta \ge 1$, Theorem~1 in \cite{Guo-Qi-Srivastava2007-02.tex} said that the function $f_{\alpha,\beta,+1}(x)$ is logarithmically completely monotonic on $(0,\infty)$ if and only if $\alpha\le\frac12$, this means that inequalities in \eqref{reason-ineq-1} and \eqref{reason-ineq-2} are reversed, and so inequalities in \eqref{beta>1/2-dou-ineq=1} and \eqref{beta>1/2-dou-ineq=2} are valid.
\par
If $\beta>0$ and $\alpha\le0$, Theorem~1 in \cite{Guo-Qi-Srivastava2007-02.tex} also said that the function $f_{\alpha,\beta,+1}(x)$ is logarithmically completely monotonic on $(0,\infty)$, this means that inequalities in \eqref{reason-ineq-1} and \eqref{reason-ineq-2} are also reversed, and so inequalities \eqref{beta>0-1} and \eqref{beta>0-2} are valid.
\par
When $\beta=0$, several mathematicians have proved that the functions $f_{\alpha,0,+1}(x)$ and $f_{\alpha,0,-1}(x)$ are logarithmically completely monotonic on $(0,\infty)$ if and only if $\alpha\le\frac12$ and $\alpha\ge1$ respectively, which implies by reasoning as above the double inequalities \eqref{qi-psi-ineq-1} and \eqref{qi-psi-ineq}. The proof of Theorem~\ref{kevic-type-ineq} is complete.
\end{proof}

\begin{proof}[Proof of Theorem~\ref{gurland-deriv-thm}]
The first conclusion in \cite[Theorem~1]{Guo-Qi-Srivastava2007-02.tex} implies that the function $f_{\alpha,\beta,+1}(x)$ is logarithmically convex for $\beta>0$ and $\alpha\le0$ on $(0,\infty)$. Combining this with Jensen's inequality for convex functions yields
\begin{equation}\label{jensen-discreate}
\ln\frac{\exp\bigl(\sum_{k=1}^np_kx_k\bigr)\Gamma\bigl(\sum_{k=1}^np_kx_k+\beta\bigr)} {\bigl(\sum_{k=1}^np_kx_k\bigr)^{\sum_{k=1}^np_kx_k+\beta-\alpha}} \le\sum_{k=1}^np_k\ln\frac{\exp(x_k)\Gamma(x_k+\beta)}{x_k^{x_k+\beta-\alpha}},
\end{equation}
where $n\in\mathbb{N}$, $x_k>0$ for $1\le k\le n$, $p_k\ge0$ satisfying $\sum_{k=1}^np_k=1$, $\beta>0$ and $\alpha\le0$. Rearranging it leads to the inequality \eqref{n-gurland-ineq}.
\par
The third conclusion in \cite[Theorem~1]{Guo-Qi-Srivastava2007-02.tex} implies that the function $f_{\alpha,\beta,+1}(x)$ is also logarithmically convex for $\beta\ge1$ and $\alpha\le\frac12$ on $(0,\infty)$, hence, the inequality~\eqref{jensen-discreate} is also valid for $\beta\ge1$ and $\alpha\le\frac12$.
\par
Theorem~\ref{fth3} above implies that the function $f_{\alpha,\beta,+1}(x)$ is logarithmically concave for $\alpha\ge\beta\ge\frac12$ on $(0,\infty)$, therefore, the inequality~\eqref{jensen-discreate} is reversed. Theorem~\ref{gurland-deriv-thm} is proved.
\end{proof}

\begin{proof}[Proof of Theorem~\ref{gurland-deriv-thm-mon}]
Theorem~\ref{fth3} implies that the function $f_{\alpha,\beta,-1}(x)$ is decreasing on $(0,\infty)$ if $\alpha\ge\beta\ge\frac12$, this is,
\begin{equation*}
\frac{e^y\Gamma(y+\beta)}{y^{y+\beta-\alpha}}>\frac{e^x\Gamma(x+\beta)}{x^{x+\beta-\alpha}}
\end{equation*}
for $y>x>0$, which can be rearranged as
\begin{gather*}
\frac{\Gamma(y+\beta)}{\Gamma(x+\beta)}>e^{x-y}\frac{y^{y+\beta-\alpha}}{x^{x+\beta-\alpha}},\\
\biggl[\biggl(\frac{y}{x}\biggr)^{\alpha-\beta}\frac{\Gamma(y+\beta)}{\Gamma(x+\beta)}\biggr]^{1/(y-x)} >\frac1e\biggl(\frac{y^{y}}{x^{x}}\biggr)^{1/(y-x)}.
\end{gather*}
The proof of Theorem~\ref{gurland-deriv-thm-mon} is complete.
\end{proof}

\end{document}